\documentclass[12pt]{amsart}
\usepackage[a4paper, margin=1.25in]{geometry}
\makeatletter
\@namedef{subjclassname@2020}{%
  \textup{2020} Mathematics Subject Classification}
\makeatother
\title[Hilbert Property]{Hilbert Property for Double Conic Bundles and del Pezzo Varieties}
\author{Sam Streeter}
\address{Department of Mathematical Sciences, University of Bath, Claverton Down, Bath, BA2 7AY, United Kingdom}
\email{sjbs20@bath.ac.uk}
\usepackage{amsmath}
\usepackage{amssymb}
\usepackage{amsthm}
\usepackage{enumitem}
\usepackage{tikz-cd}
\usepackage{soul}
\usepackage[linktocpage=true]{hyperref}
\usepackage{mathrsfs}
\theoremstyle{definition}
\newtheorem{mydef}{Definition}
\newtheorem{note}[mydef]{Note}
\newtheorem{remark}[mydef]{Remark}

\theoremstyle{plain}
\newtheorem{theorem}[mydef]{Theorem}
\newtheorem{proposition}[mydef]{Proposition}
\newtheorem{lemma}[mydef]{Lemma}
\newtheorem{corollary}[mydef]{Corollary}
\newtheorem{conjecture}[mydef]{Conjecture}
\numberwithin{mydef}{section}
\let\originalleft\left
\let\originalright\right
\renewcommand{\left}{\mathopen{}\mathclose\bgroup\originalleft}
\renewcommand{\right}{\aftergroup\egroup\originalright}
\DeclareMathOperator{\ch}{char}

\DeclareMathOperator{\Gal}{Gal}

\begin{document}
\begin{abstract}
In this paper we prove that, over a Hilbertian ground field, surfaces with two conic fibrations whose fibres have non-zero intersection product have the Hilbert property. We then give an application of this result, namely the verification of the Hilbert property for certain del Pezzo varieties.
\end{abstract}
\maketitle
\setcounter{tocdepth}{1}
\tableofcontents
\section{Introduction} \label{sectionI}
\subsection{Motivation} \label{sectionM}
A common goal in algebraic geometry is to describe the abundance of rational points on a given algebraic variety. In particular, when a variety possesses infinitely many rational points, there are various notions to describe their abundance, which are either of a quantitative, analytic flavour, or of a qualitative, topological one. In the latter scenario, we often consider density in terms of certain topologies on the variety, the most natural being the Zariski topology.

Here, we describe and study a notion of abundance which is stronger than Zariski density, namely the \emph{Hilbert property}: briefly, a variety $X/k$ is said to have the Hilbert property if the set $X\left(k\right)$ is not thin (we give the definition of thin sets in Section \ref{sectionB}). In fact, this particular notion of abundance of rational points is connected to the so-called \emph{inverse Galois problem}, which is to determine whether every finite group $G$ is realisable as the Galois group of a Galois extension of a given field, by the following conjecture of Colliot-Th\'el\`ene and Sansuc.
\begin{conjecture} \emph{(\cite[p.~190]{PHSUFTA})} \label{conjecture}
Let $X$ be a unirational variety over a number field. Then $X$ has the Hilbert property.
\end{conjecture}
As noted in the proof of \cite[Thm.~3.5.9,~p.~30]{TGT}, the truth of this conjecture implies a positive solution to the inverse Galois problem over $\mathbb{Q}$. 

The Hilbert property was investigated by Corvaja and Zannier in \cite{OHPFGAV}, where the authors proved that the Fermat quartic $$x^4 + y^4 = z^4 + w^4 \subset \mathbb{P}^3$$ over $\mathbb{Q}$ has the Hilbert property. This surface is not geometrically unirational, hence the converse of Conjecture \ref{conjecture} is false. Demeio later proved in \cite{NRVWHP} and \cite{EFHP} that, over a number field $k$, certain double elliptic surfaces and all cubic hypersurfaces with a $k$-rational point have the Hilbert property. Further, it follows from the work of Salberger and Skorobogatov in \cite{WASDTQF} that smooth intersections of two quadrics with a rational point over a number field have the Hilbert property. Fibrations play an important role in proving these results, and the role of fibrations in studying the Hilbert property is emphasised by the main theorem of Bary-Soroker, Fehm and Petersen in \cite{OVHT} (Theorem \ref{ovhttheorem}), which allows us to reduce questions about the rational points of a variety to the fibres of some morphism from it. This result is crucial to the methods in this paper, where we employ conic fibrations. Conic fibrations were also used by Coccia in \cite{HPIPASCS} to prove that affine smooth cubic surfaces over number fields satisfy a version of the Hilbert property for $S$-integral points after a finite field extension.
\subsection{Results} \label{sectionR}
The main result of this paper is the following theorem.
\begin{theorem} \label{thm1}
Let $S$ be a smooth projective surface defined over a Hilbertian field $k$ possessing two conic fibrations $\pi_i: S \rightarrow \mathbb{P}^1$, $i=1,2$ such that two fibres from different fibrations have non-zero intersection product, and suppose that there exists $P_0 \in S\left(k\right)$ lying on a smooth fibre of one of the conic fibrations. Then $S$ has the Hilbert property.
\end{theorem}
\begin{remark}
We observe that these surfaces are \emph{geometrically rational}, i.e.\ they are rational surfaces over the algebraic closure of their field of definition.
\end{remark}
In particular, we obtain the following result for del Pezzo surfaces.
\begin{theorem} \label{thm2}
Let $S$ be a del Pezzo surface of degree $d$ defined over a Hilbertian field $k$ with $S\left(k\right) \neq \emptyset$. Suppose that one of the following holds:
\begin{enumerate} [label=(\alph*)]
\item $d \geq 4$;
\item $d=3$ and there exists a line $L$ on $S$ (under its anticanonical embedding);
\item $d \in \{1,2\}$ and $S$ possesses a conic fibration.
\end{enumerate} 
Then $S$ has the Hilbert property.
\end{theorem}
By induction with Theorem \ref{thm2} as the base case, we will also prove the following result for del Pezzo varieties (which we define in Section \ref{sectionDPV}).
\begin{theorem} \label{thm3}
Let $\left(X,H\right)$ be a smooth del Pezzo variety of degree $d$ and dimension $n \geq 2$ defined over a Hilbertian ground field $k$ with $X\left(k\right) \neq \emptyset$. Suppose that one of the following holds:
\begin{enumerate} [label=(\alph*)]
\item $d \geq 4$;
\item $d=3$ and there exists a line $L$ on $X$ (under its anticanonical embedding).
\end{enumerate}
Then $X$ has the Hilbert property. 
\end{theorem}
Examples of del Pezzo varieties include cubic hypersurfaces (degree $3$) and intersections of two quadric hypersurfaces (degree $4$).
\subsection{Conventions} \label{subsectionC}
Throughout this paper, we will work over an arbitrary ground field $k$ of characteristic $0$, and all algebraic varieties shall be understood to be integral over the ground field $k$ and quasi-projective.

We include the following definition for clarity.
\begin{mydef}
Let $S$ be a smooth projective surface. A \emph{conic fibration} of $S$ is a morphism $\pi: S \rightarrow \mathbb{P}^1$ such that all fibres of $\pi$ are isomorphic to plane conics. We call a surface with a conic fibration a \emph{conic bundle} (over $\mathbb{P}^1$).  
\end{mydef}
\begin{remark}
If $\pi: S\rightarrow \mathbb{P}^1$ is a conic fibration, then by the adjunction formula (see \cite[Lem.~2.1,~p.~412]{RPBHGCBS}), the fibre $\pi^{-1}\left(P\right)$ is reduced for all $P \in \mathbb{P}^1$.
\end{remark}
\section{Background} \label{sectionB}
We now define thin sets and the Hilbert property. We recall our conventions that $\ch k = 0$ and that algebraic varieties are integral and quasi-projective. We begin by giving the definition of thin sets from \cite[Ch.~3,~\S1]{TGT}, amending the notation and terminology slightly.
\begin{mydef}
Let $V$ be a variety over $k$.

A subset $A \subset V\left(k\right)$ is of \emph{type I} if there is a proper Zariski closed subset $W \subset V$ with $A \subset W\left(k\right)$. 

A subset $A \subset V\left(k\right)$ is of \emph{type II} if there is a variety $V'$ with $\textrm{dim}V' = \textrm{dim}V$ and a generically finite dominant morphism $\phi: V' \rightarrow V$ of degree $\geq 2$ with $A \subset \phi\left(V'\left(k\right)\right)$.

A subset $A \subset V\left(k\right)$ is \emph{thin} if it is contained in a finite union of subsets of $V\left(k\right)$ of types I and II.  
\end{mydef}
\begin{remark}
Serre mentions (see \cite[p.~20]{TGT}) that the variety $V'$ in the definition of type II thin sets can be taken to be geometrically irreducible, as otherwise we have $V'\left(k\right) \subset W'\left(k\right)$ for some Zariski closed proper subset $W'$ of $V'$.
\end{remark}
\begin{remark}
Further to the above remark, we can ask that $V'$ be normal and that $\phi$ be finite. Indeed, given a variety $V'$ and a generically finite dominant morphism $\phi: V' \rightarrow V$, the function field $K\left(V'\right)$ of $V'$ is a finite extension of the function field $K\left(V\right)$ of $V$, and we can consider the \emph{normalisation of $V$ in $K\left(V'\right)$} (see \cite[~Def.\,4.1.24,~p.~120]{AGAC}). This is a normal scheme $\widetilde{V'}$ with $K(\widetilde{V'}) = K\left(V'\right)$ together with an integral morphism $\widetilde{\phi}: \widetilde{V'} \rightarrow V$. By \cite[~Prop.~4.1.27,~p.~121]{AGAC}, we have that $\widetilde{\phi}$ is finite and $\widetilde{V'}$ is an algebraic variety. Moreover, by the universal property of normalisation, $\phi: V' \rightarrow V$ factors uniquely through $\widetilde{\phi}$:
\begin{equation*}
\begin{tikzcd}
& & \widetilde{V'} \arrow[d, "\widetilde{\phi}"] \\
& V' \arrow[r, "\phi"] \arrow[ru, dotted, "\exists !"] & V.
\end{tikzcd}
\end{equation*}
In particular, we see that $\phi\left(V'\left(k\right)\right) \subset \widetilde{\phi}(\widetilde{V'}\left(k\right))$, so replacing the pair $\left(V', \phi\right)$ by $(\widetilde{V'}, \widetilde{\phi})$ allows us to demand these stronger requirements for type II thin sets.
\end{remark}
\begin{mydef} \label{hpdef}
Let $V$ be an algebraic variety over a field $k$. We say that $V$ has the \emph{Hilbert property} (over $k$) if the set $V\left(k\right)$ is not thin.
\end{mydef}
\begin{note}
If $V$ has the Hilbert property over $k$, then $V\left(k\right)$ is Zariski dense. Also, the Hilbert property is a birational invariant.
\end{note}
\begin{mydef}
We say that a field $k$ of characteristic zero is \emph{Hilbertian} if there exists a variety $V/k$ such that $V$ has the Hilbert property over $k$.
\end{mydef}
\begin{remark} \label{hpremark}
It can be shown that the field $k$ is Hilbertian if and only if $\mathbb{P}^1$ has the Hilbert property over $k$ (see \cite[Exercise~3.1.1,~p.~20]{TGT}).
\end{remark}
We will also make use of the following well-known result for fibre products of curves over $\mathbb{P}^1$, a proof of which we give for completeness.
\begin{lemma} \label{fplemma}
Let $C$ and $D$ be two regular, geometrically irreducible curves over a field $k$ equipped with non-constant morphisms $\phi: C \rightarrow \mathbb{P}^1$ and $\psi: D \rightarrow \mathbb{P}^1$ having disjoint branch loci. Then $C \times_{\mathbb{P}^1} D$ is a regular, geometrically irreducible curve.
\end{lemma}
\begin{proof}
Since $\phi$ and $\psi$ have disjoint branch loci, there is an open neighbourhood for any point in $\mathbb{P}^1$ over which one of $\phi$ and $\psi$ is \'etale and the other is regular, so the corresponding open set in $C \times_{\mathbb{P}^1} D$ is \'etale over a regular curve (since the base change of an \'etale morphism is \'etale), hence $C \times_{\mathbb{P}^1} D$ is regular.

Since $\overline{k}\left(C \times_{\mathbb{P}^1} D \right) = \overline{k}\left(C\right) \otimes_{\overline{k}\left(\mathbb{P}^1\right)} \overline{k}\left(D\right)$, the curve $C \times_{\mathbb{P}^1} D$ is geometrically irreducible if and only if $\overline{k}\left(C\right)$ and $\overline{k}\left(D\right)$ are linearly disjoint. By \cite[Lem.~2.5.3,~p.~35]{FA}, it suffices to check linear disjointness of $\overline{k}\left(C\right)$ and $\overline{k}\left(D^{\Gal}\right)$, where $\psi^{Gal}: D^{\Gal} \rightarrow \mathbb{P}^1$ denotes the Galois closure of $\psi: D \rightarrow \mathbb{P}^1$ (see \cite[p.~20]{TGT}), which holds if and only if $\overline{k}\left(C\right) \cap \overline{k}\left(D^{\Gal}\right) = \overline{k}\left(\mathbb{P}^1\right)$. Suppose that $\overline{k}\left(C\right)$ and $\overline{k}\left(D^{\Gal}\right)$ possess a common subextension with corresponding curve and morphism $\theta: E \rightarrow \mathbb{P}^1$. Since $\mathbb{P}^1_{\overline{k}}$ is algebraically simply connected, $\theta$ is ramified. Any branch points of $\theta$ are common branch points of $\phi$ and $\psi^{\Gal}$, but the branch points of $\psi^{Gal}$ are the same as those of $\psi$, hence $\theta$ is unramified, contradiction, so $C \times_{\mathbb{P}^1} D$ is geometrically irreducible.
\end{proof}
\section{Double conic bundles and del Pezzo surfaces}
We now prove Theorem \ref{thm1}. The following result shall be essential to our proof.
\begin{theorem} \label {ovhttheorem} \cite[Thm.~1.1,~p.~1894]{OVHT}
Let $k$ be a field and $f : X \rightarrow S$ a dominant morphism of $k$-varieties. Assume that the set of $s \in S\left(k\right)$ for which the fibre $f^{-1}\left(s\right)$ has the Hilbert property is not thin. Then $X$ has the Hilbert property.
\end{theorem}
In particular, we have the following corollary for conic fibrations.
\begin{corollary} \label{maincor}
Let $k$ be a Hilbertian field and $\pi: S \rightarrow \mathbb{P}^1$ a conic fibration. Assume that the set
\begin{equation*}
\{P \in \mathbb{P}^1\left(k\right) : \pi^{-1}\left(P\right)\left(k\right) \neq \emptyset \}
\end{equation*}
is not thin. Then $S$ has the Hilbert property over $k$.
\end{corollary}
\begin{proof}
By the Riemann-Roch theorem, any smooth curve of genus zero with a rational point is isomorphic to $\mathbb{P}^1$, which has the Hilbert property (see Remark \ref{hpremark}), and all but finitely many fibres of $\pi$ are smooth curves of genus zero.
\end{proof}
Then Theorem \ref{thm2} follows immediately from the following proposition.
\begin{proposition} \label{mainprop}
Let $S$ be a smooth projective surface defined over a Hilbertian field $k$ with two distinct conic fibrations $\pi_i: S \rightarrow \mathbb{P}^1$, $i=1,2$, and suppose that there exists $P_0 \in S\left(k\right)$ such that $\pi_1^{-1}\left(\pi_1\left(P_0\right)\right)$ is smooth. For each $i =1 ,2$, the set
\begin{equation*}
A_i =\{P \in \mathbb{P}^1\left(k\right) : \pi_i^{-1} \left(P\right)\left(k\right) \neq \emptyset\} 
 \end{equation*}
is not thin.
\end{proposition}
\begin{lemma} \label{mainlemma}
Each of the sets $A_i$ in Proposition \ref{mainprop} is infinite. That is, each of the conic fibrations has infinitely many fibres which each have a rational point.
\end{lemma}
\begin{proof}
Since $\pi_1^{-1}\left(\pi_1\left(P_0\right)\right)$ is smooth, $\#\pi_1^{-1} \left(\pi_1\left(P_0\right)\right)\left(k\right) = \infty$. Let $\{Q_i\}_{i \in \mathbb{N}}$ be an infinite collection of rational points on $\pi_1^{-1} \left(\pi_1 \left(P_0\right) \right)$. Since each smooth fibre of one fibration intersects any fibre from the other fibration in a finite set, $\{\pi_2\left(Q_i\right)\} \subset A_2$ is infinite, so $A_2$ is infinite. Repeating this argument with any one of the fibres $\pi_2^{-1}\left(\pi_2\left(Q_i\right)\right)$ which is smooth, we see that $A_1$ is also infinite. 
\end{proof}
\begin{proof}[Proof of Proposition \ref{mainprop}]
It suffices to prove that a subset of $A_2$ is not thin, since Lemma \ref{mainlemma} implies that $\pi_2$ also has a smooth fibre with a rational point. So, assume that $\bigcup_{P \in A_1} \pi_2 \left(\pi_1^{-1}\left(P\right)\left(k\right)\right) \subset A_2$ is thin. Then it is easily seen that there exists a finite collection of dominant finite morphisms $\varphi_i : C_i \rightarrow \mathbb{P}^1$, $i=1,...,r$ where $C_i$ is a normal (hence smooth) geometrically irreducible curve defined over $k$ and $\deg \varphi_i > 1$, such that $\bigcup_{P \in A_1} \pi_2 \left(\pi_1^{-1}\left(P\right)\left(k\right)\right) \subset \bigcup_{i=1}^r \varphi_i \left(C_i\left(k\right)\right)$. Let
\begin{equation*}
B = \bigcup_{i=1}^r B_{\varphi_i} = \{R_1,\ldots,R_m\},
\end{equation*}
where $B_{\varphi_i}$ denotes the branch locus of the morphism $\varphi_i$.

Since each fibre $\pi_2^{-1}\left(R_i\right)$ is reduced, the morphism $$\pi_1|_{\pi_2^{-1}\left(R_i\right)}: \pi_2^{-1}\left(R_i\right) \rightarrow \mathbb{P}^1$$ has finite branch locus for each $i=1,\dots,m$. Denote the union of these branch loci by $B'$, and denote by $A_1'$ the set of $P \in A_1 \setminus B'$ such that $\pi_1^{-1}\left(P\right)$ is smooth, which is non-empty by Lemma \ref{mainlemma}. Taking $P \in A_1'$, the morphism $\pi_1|_{\pi_2^{-1}\left(R_i\right)}: \pi_2^{-1}\left(R_i\right) \rightarrow \mathbb{P}^1$ is unramified over $P$ for each $i$, which is equivalent to $\left(\pi_1|_{\pi_2^{-1}\left(R_i\right)}\right)^{-1}\left(P\right) = \pi_1^{-1}\left(P\right) \cap \pi_2^{-1}\left(R_i\right)$ being reduced for each $i$ (see \cite[Lem.~4.3.20,~p.~139]{AGAC}) and in turn equivalent to $\pi_2|_{\pi_1^{-1}\left(P\right)}$ being unramified over every $R_i$.
Then for all $P \in A_1'$, $i =1,...,r$, the branch loci of $\pi_2|_{\pi_1^{-1}\left(P\right)}$ and $C_i$ are disjoint.

For each $P \in A_1'$, consider the fibre product
\begin{equation*} \label{FP}
\begin{tikzcd}
& F_i \arrow[r, "\psi_i"] \arrow[d, "\theta_i"] & C_i \arrow[d, "\varphi_i"] \\
& \pi_1^{-1}\left(P\right) \arrow[r, "\pi_2|_{\pi_1^{-1}\left(P\right)}"] & \mathbb{P}^1. 
\end{tikzcd}
\end{equation*}
By Lemma \ref{fplemma}, $F_i$ is a smooth, irreducible curve. Since surjectivity is preserved under base change, the maps $\psi_i$ and $\theta_i$ are surjective. By comparing ramification indices of points on $\mathbb{P}^1$ for the morphism $\varphi_i \circ \psi_i = \pi_2|_{\pi_1^{-1}\left(P\right)} \circ \theta_i: F_i \rightarrow \mathbb{P}^1$ and using the compatibility of ramification indices with composition of morphisms, one sees that $\theta_i$ is ramified, hence $\deg \theta_i > 1$. Since $\pi_1^{-1}\left(P\right)$ has the Hilbert property, the set $\pi_1^{-1}\left(P\right)\left(k\right) \setminus \bigcup_{i=1}^r \theta_i \left(F_i \left(k\right)\right)$ is infinite.

Let $Q \in \pi_1^{-1}\left(P\right)\left(k\right) \setminus \bigcup_{i=1}^r \theta_i \left(F_i \left(k\right)\right)$. Then we claim that 
\begin{equation*}
\pi_2 \left(Q\right) \in \bigcup_{P \in A_1} \pi_2 \left(\pi_1^{-1}\left(P\right) \left(k\right) \right) \setminus \bigcup_{i=1}^r \varphi_i \left(C_i\left(k\right)\right).
\end{equation*}
Clearly $\pi_2 \left(Q\right) \in \bigcup_{P \in A_1} \pi_2\left(\pi_1^{-1}\left(P\right) \left(k\right)\right)$, and if there is $Q' \in C_i\left(k\right)$ with $\varphi_i \left(Q'\right) = \pi_2 \left(Q\right)$ for some $i$, then there exists $Q'' \in F_i\left(k\right)$ such that $\theta_i \left(Q''\right) = Q$, which is impossible since $Q \in \pi_1^{-1}\left(P\right)\left(k\right) \setminus \theta_i \left(F_i \left(k\right)\right)$. Then we have a contradiction to our original assumption that $\bigcup_{P \in A_1} \pi_2\left(\pi_1^{-1}\left(P\right)\left(k\right)\right) \subset \bigcup_{i=1}^r \varphi_i \left(C_i \left(k\right)\right)$, and so $\bigcup_{P \in A_1} \pi_2\left(\pi_1^{-1}\left(P\right)\left(k\right)\right)$ is not thin.
\end{proof}
\begin{proof}[Proof of Theorem \ref{thm1}]
The result follows from Proposition \ref{mainprop} and Corollary \ref{maincor}.
\end{proof}
For the proof of Theorem \ref{thm2}, we will employ the following fact, which is seen in the proof of \cite[Thm.~5,~p.~21]{MMRS}. Although the del Pezzo surfaces in the statement there are minimal, minimality is not necessary. We give the proof for completeness.
\begin{lemma} \label{dcslemma}
Let $S$ be a del Pezzo surface of degree $d \in \{1,2,4\}$ with a conic fibration $\pi: S \rightarrow \mathbb{P}^1$. Then there exists a second conic fibration $\pi': S \rightarrow \mathbb{P}^1$ such that fibres from distinct fibrations have non-zero intersection product.
\end{lemma}
\begin{proof}
Let $C$ be a fibre of the conic fibration $\pi: S \rightarrow \mathbb{P}^1$, and define $$D := - \frac{4}{d} K_S - C.$$ We will show that the linear system $|D|$ induces the sought conic fibration $\pi'$.

By the Riemann-Roch theorem for surfaces \cite[Thm.~V.1.6,~p.~362]{AG}, $$l\left(D\right) - s\left(D\right) + l\left(K_S - D\right) = \frac{1}{2} D \cdot \left(D - K_S\right) + \chi\left(S, \mathcal{O}_S\right).$$  Since $S$ is geometrically rational and $-K_S$ is ample, it follows from Serre duality \cite[Cor.~III.7.7,~p.~244]{AG} that $\chi\left(S, \mathcal{O}_S\right) = 1$. Further, if $E$ is an effective divisor linearly equivalent to $K_S - D$, then the ampleness of $-K_S$ implies that $-K_S \cdot E > 0$, but $$-K_S \cdot E = -K_S \cdot \left(\left(1 + \frac{4}{d}\right) K_S + C\right) =  - \left(1 + \frac{4}{d}\right)d + K_S \cdot C = -\left(d+2\right) < 0,$$ where $K_S \cdot C = - 2$ by the adjunction formula \cite[Exercise~V.1.3,~p.~366]{AG} and the fact that $C^2 = 0$. Thus $l\left(K_S - D\right) = 0$. Next, we compute $$D^2 = \frac{16}{d^2}d + \frac{8}{d} \left(-2\right) + 0 = 0,$$ $$-K_S \cdot D = \frac{4}{d}d - 2 = 2.$$ Then the Riemann-Roch theorem applied to $D$ rearranges to give $$l\left(D\right) = 2 + s\left(D\right) \geq 2,$$ and so $D$ is effective. By the adjunction formula, $D$ has arithmetic genus zero.

It only remains to verify that $|D|$ is base-point-free. Since $-K_S \cdot D = 2$ and $-K_S$ is ample, any element of $|D|$ has at most $2$ irreducible components. Then, for any $E \in |D|$, one of the following possibilities holds:
\begin{enumerate} [label=(\alph*)]
\item $E$ is irreducible;
\item $E = E_1 + E_2$, where $E_1 \neq E_2$, or
\item $E = 2E_1$
\end{enumerate}
for some irreducible curves $E_1$, $E_2$ on $S$. In (c), we obtain $-K_S \cdot E = 1$, but then the adjunction formula implies that the arithmetic genus of $E_1$ is not an integer, hence this possibility cannot be realised. In (b), we square to get $$0 = E_1^2 + E_2^2 + 2E_1 \cdot E_2 \geq E_1^2 + E_2^2.$$ We cannot have $E_i^2 = 0$ for $i=1,2$, as then the adjunction formula implies that the arithmetic genus of $E_i$ is not an integer, and on a del Pezzo surface, every effective divisor has self-intersection at least $-1$, hence $E_1^2 = E_2^2 = -1$. However, there are only finitely many curves on $S$ (namely the exceptional curves) with self-intersection $-1$, hence all but finitely many curves in $|D|$ are irreducible. Since the intersection product of two irreducible curves is zero if and only if they do not intersect, $|D|$ is base-point-free, hence it gives rise to a morphism $\pi': S \rightarrow \mathbb{P}^1$. Since the arithmetic genus of $D$ is zero, $\pi'$ is a conic fibration.

Finally, since the intersection product of a fibre from $\pi$ with a fibre from $\pi'$ is $$C \cdot D = C \cdot \left(-\frac{4}{d}K_S - C\right) = \frac{8}{d} > 0,$$ we are done. 
\end{proof} 
\begin{proof} [Proof of Theorem \ref{thm2}]
The result for $d \geq 5$ follows immediately from the fact that del Pezzo surfaces of degree $d \geq 5$ are rational, see \cite[Thm.~9.4.8,~p.~277]{RPV}. If $d=4$, then the assumption $S\left(k\right) \neq \emptyset$ implies that there exists $P \in S\left(k\right)$ not lying on any of the exceptional curves of $S$, see \cite[Thm.~30.1,~p.~162]{CF}. Blowing up $S$ at $P$, we obtain a del Pezzo surface $\widetilde{S}$ of degree $3$ with a curve $L$ which is a line under the anticanonical embedding, and $L\left(k\right) \neq \emptyset$ implies $\widetilde{S}\left(k\right) \neq \emptyset$, so part (b) follows from part (c).

For part (c), it can be shown that $S$ has a conic fibration (considering the residual intersection of the hyperplanes containing $L$), and by \cite[Thm.~30.1,~p.~162]{CF}, there exists a rational point $P \in S\left(k\right)$ not lying on any of the exceptional curves of $S$. Blowing up $S$ at $P$, we obtain a del Pezzo surface of degree $2$ with a conic fibration and a rational point, so part (c) follows from part (d).

By \cite[Cor.~8,~p.~919]{QFECUD1CB} and Lemma \ref{dcslemma}, a del Pezzo surface $S$ of degree $d \in \{1,2\}$ with a rational point and a conic fibration satisfies the hypotheses of Theorem \ref{thm1}, hence it has the Hilbert property, hence part (d) holds.
\end{proof}
\begin{remark} \label{cbrem}
By a result of Iskovskih (see \cite[Thm.~1.6,~p.~572]{RSPRC}), the canonical class of a relatively minimal double conic bundle has positive self-intersection.
\end{remark}
\section{Del Pezzo varieties} \label{sectionDPV}
In this section, we introduce smooth del Pezzo varieties and prove Theorem \ref{thm3}. We draw from the excellent exposition given in \cite[Ch.~3]{AG5}.
\begin{mydef}
A \emph{smooth del Pezzo variety} is a pair $\left(X,H\right)$ consisting of a smooth projective algebraic variety $X$ and an ample Cartier divisor $H$ on $X$ such that $-K_X = \left(n - 1\right) H$, where $n = \dim X$. The \emph{degree} of a smooth del Pezzo variety $\left(X,H\right)$ of dimension $n$ is defined by $\deg X := H^n$. 
\end{mydef}
 \begin{proposition} \cite[Prop.~3.2.4(ii),~p.~54]{AG5}
 Let $\left(X, H\right)$ be a smooth del Pezzo variety. Put $n:= \dim X$ and $d := \deg X$. If $d \geq 3$, then the linear system $|H|$ determines an embedding $\phi_{|H|}: X \hookrightarrow \mathbb{P}^{n+d-2}$.
 \end{proposition}
 \begin{note}
 This result tells us that, for $d \geq 3$, we can identify the divisor class $H$ with the class of hyperplane sections under the embedding induced by $|H|$.
 \end{note}
 A complete classification of del Pezzo varieties of dimension $n \geq 3$ arising from the work of Fujita and Iskovskikh can be found in \cite[Thm.~3.3.1,~p.~55]{AG5}.
 \begin{proof}[Proof of Theorem \ref{thm3}]
We induct on $n = \dim X$. The base case $n=2$ is Theorem \ref{thm2}, so assume that $n \geq 3$.

First suppose that $d \geq 4$. 
Identify $X$ with its image as a projective variety in $\mathbb{P}^{n+d-2}$ under the embedding induced by $|H|$. Let $P \in X\left(k\right)$, and denote by $\Lambda\left(P\right)$ the linear system of divisors $D \in |H|$ with $P \in D$. This linear system is of dimension $n+d-3$, corresponding to a hyperplane in $\mathbb{P}\left(\mathscr{L}\left(H\right)\right)$.

Now, an element $D \in \Lambda\left(P\right)$ is a hyperplane section of $X$ through $P$, and by Bertini's theorem, a general element of $\Lambda\left(P\right)$ is smooth away from $P$, hence it is smooth if and only if the corresponding hyperplane is not tangent to $X$ at $P$. Since $X$ is smooth, there is a unique hyperplane tangent to $X$ at $P$, so we conclude that the general element of $\Lambda\left(P\right)$ is smooth. Further, note that, for a smooth divisor $D \in |H|$, the adjunction formula in the form $K_D = D \cdot \left(K_X + D\right)$ gives us $K_D = -\left(n-2\right)D \cdot D$, so $\left(D, H|_D\right)$ is a smooth del Pezzo surface of dimension $n-1$ and degree $\left(H|_D\right)^{n-1} = H \cdot H^{n-1} = H^n = d$. Thus the smooth elements of $|H|$ are del Pezzo varieties of dimension $n-1$ and degree $d$.

Choose two smooth elements $D_1, D_2 \in \Lambda\left(P\right)$, and denote by $\Pi\left(P\right)$ the pencil generated by $D_1$ and $D_2$. Let $f: X \dashrightarrow \mathbb{P}^1$ denote the rational map induced by the linear system $\Pi\left(P\right)$. Note that the general element of $\Pi\left(P\right)$ is smooth, i.e.\ the pencil has only finitely many singular members, since the smooth divisors in $\Lambda\left(P\right)$ correspond to a dense open subset and $\Pi\left(P\right)$ corresponds to a line intersecting the aforementioned open subset. By resolution of singularities (see \cite[Main~Thm.~II,~p.~142]{RSAVOFCZ}), we obtain a variety $\widetilde{X}$ birational to $X$ and a morphism $\phi: \widetilde{X} \rightarrow \mathbb{P}^1$ extending $f$, whose smooth fibres are strict transforms of smooth elements of $\Pi\left(P\right)$ under the associated blow-up.

Now, a divisor is birational to its strict transform under a blow-up, and the Lang--Nishimura theorem \cite[Thm.~3.6.11,~p.~92]{RPV} states that having a smooth $k$-point is a birational invariant, hence all but finitely many of the fibres of $\phi$ (i.e.\ all those fibres corresponding to the strict transform of a smooth divisor of $\Pi\left(P\right)$) are birational to a smooth del Pezzo variety of dimension $n-1$ and degree $d \geq 4$ with a rational point. By the inductive hypothesis, these fibres have the Hilbert property, hence, by Theorem \ref{ovhttheorem}, we see that $X$ has the Hilbert property.

It only remains to consider the case where $d=3$ and there exists a line $L$ on $X$.
Denote by $\Lambda\left(L\right)$ the linear system of divisors $D \in |H|$ with $L \subset D$. This is a linear system of dimension $n-1 \geq 2$, corresponding to a codimension-$2$ linear variety in $\mathbb{P}\left(\mathscr{L}\left(H\right)\right)$. By Bertini's theorem, a general element of $\Lambda\left(L\right)$ is smooth away from $L$, hence it is smooth if and only if the corresponding hyperplane is not tangent to $X$ at any point $Q \in L$. Since $X$ is smooth, there is a unique hyperplane tangent to $X$ at each $Q \in L$, so the general element of $\Lambda\left(L\right)$ is smooth. As above, we conclude that the smooth divisors in $\Lambda\left(L\right)$ are del Pezzo varieties of dimension $n-1$ and degree $3$ containing the line $L$, and, defining $\Pi\left(P\right)$ to be the pencil generated by two smooth divisors in $\Lambda\left(P\right)$, the general element of this pencil is smooth. Further, since $L\left(k\right) \neq \emptyset$, the smooth elements have a rational point. Then, proceeding as above, we conclude that $X$ has the Hilbert property.
\end{proof}
\subsection*{Acknowledgements} \label{subsectionA}
I am sincerely thankful to my supervisor Daniel Loughran for his keen insights and tireless support. I would also like to thank Julian Demeio for pointing out an important flaw in the original draft of this paper. Finally, I would like to thank Jean-Louis Colliot-Th\'el\`ene for drawing my attention to the work of Salberger and Skorobogatov \cite{WASDTQF} and the observation in Remark \ref{cbrem}.

\end{document}